\documentclass[12pt]{article}

\usepackage{amsfonts,amssymb,amsmath,amsthm}
\usepackage{tikz}
\usetikzlibrary{matrix,arrows}
\usetikzlibrary{positioning}
\usetikzlibrary{fit}
\usetikzlibrary{patterns}

\newtheorem{theorem}{Theorem}
\newtheorem{lemma}[theorem]{Lemma}

\newtheorem{proposition}[theorem]{Proposition}

\long\def\symbolfootnote[#1]#2{\begingroup
\def\thefootnote{\fnsymbol{footnote}}\footnote[#1]{#2}\endgroup}

\title{On the Representability of Line Graphs}

\author{Sergey Kitaev\thanks{Supported by a grant from Iceland, Liechtenstein and Norway through the EEA Financial Mechanism. Supported and coordinated by Universidad Complutense de Madrid.}
\and
Pavel Salimov\thanks{Supported by the Russian Foundation for Basic Research grant no.\ 10-01-00616 and by the Icelandic Research Fund grant no.\ 090038013.}
\and
Christopher Severs\thanks{Supported by grant no.\ 090038012 from the Icelandic Research Fund.}
\and
Henning \'Ulfarsson\thanks{Supported by grant no.\ 090038011 from the Icelandic Research Fund.}}

\begin{document}

\maketitle

\begin{abstract}

A graph $G=(V,E)$ is representable if there exists a word $W$ over the alphabet $V$ such that letters $x$ and $y$ alternate in $W$ if and only if $(x,y)$ is in $E$ for each $x$ not equal to $y$. The motivation to study representable graphs came from algebra, but this subject is interesting from graph theoretical, computer science, and combinatorics on words points of view.


In this paper, we prove that for $n$ greater than $3$, the line graph of an $n$-wheel is non-representable. This not only provides a new construction of non-representable graphs, but also answers an open question on representability of the line graph of the $5$-wheel, the minimal non-representable graph. Moreover, we show that for $n$ greater than $4$, the line graph of the complete graph is also non-representable. We then use these facts to prove that given a graph $G$ which is not a cycle, a path or a claw graph, the graph obtained by taking the line graph of $G$ $k$-times is guaranteed to be non-representable for $k$ greater than $3$.

\end{abstract}

\section{Introduction}

A graph $G=(V,E)$ is representable if there exists a word $W$ over
the alphabet $V$ such that letters $x$ and $y$ alternate in $W$ if
and only if $(x,y)\in E$ for each $x\neq y$. Such a $W$ is called a {\em word-representant} of $G$. Note that in this paper we use the term graph to mean a finite, simple graph, even though the definition of representable is applicable to more general graphs. 

It was shown by Kitaev and Pyatkin, in \cite{KP}, that if a graph is representable by $W$, then one can assume that $W$ is {\em uniform}, that is, it contains the same number of copies of each letter. If the number of copies of each letter in $W$ is $k$, we say that $W$ is {\em $k$-uniform}. For example, the graph to the left in Fig.~\ref{fig:petersen} can be represented by the $2$-uniform word $12312434$
(in this word every pair of letters alternate, except $1$ and $4$, and
$2$ and $4$), while the graph to the right, the Petersen graph, can be
represented by the $3$-uniform word $027618596382430172965749083451$
(the Petersen graph cannot be represented by a $2$-uniform word as
shown in \cite{HKP2})

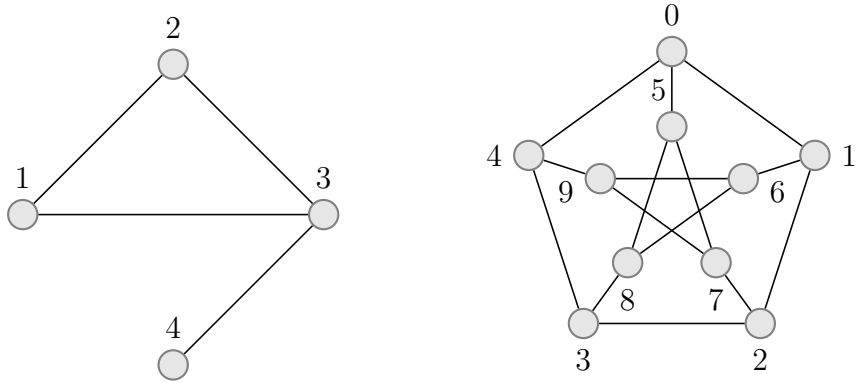
\begin{figure}[h]
\begin{center}
 \begin{tikzpicture}
 [place/.style = {circle,draw = black!50,fill = gray!20,thick}]

 \node at (0,2) [place] (2) {};
 \node [black,above] at (2.north) {$2$};

 \node at (-2,0) [place] (1) {};
 \node [black,above] at (1.north) {$1$};

 \node at (2,0) [place] (3) {};
 \node [black,above] at (3.north) {$3$};

 \node at (0,-2) [place] (4) {};
 \node [black,above] at (4.north) {$4$};

 \draw [-,semithick] (2) to (1);
 \draw [-,semithick] (2) to (3);
 \draw [-,semithick] (1) to (3);
 \draw [-,semithick] (3) to (4);

 \end{tikzpicture} \qquad\qquad
\begin{tikzpicture}[place/.style = {circle,draw = black!50,fill = gray!20,thick}, auto]

 \node [place] (0) at (90:2)      {};
 \node [black,above] at (0.north) {$0$};
 \node [place] (4) at (90+72:2)   {};
 \node [black,left] at (4.west) {$4$};
 \node [place] (3) at (90+2*72:2) {};
 \node [black,below] at (3.south) {$3$};
 \node [place] (2) at (90+3*72:2) {};
 \node [black,below] at (2.south) {$2$};
 \node [place] (1) at (90+4*72:2) {};
 \node [black,right] at (1.east) {$1$};

 \node [place] (5) at (90:1)      {};
 \node [black,above] at (5.north) [xshift = -5] {$5$};
 \node [place] (9) at (90+72:1)   {};
 \node [black,left] at (9.west) [yshift = -5] {$9$};
 \node [place] (8) at (90+2*72:1) {};
 \node [black,below] at (8.south) {$8$};
 \node [place] (7) at (90+3*72:1) {};
 \node [black,below] at (7.south) {$7$};
 \node [place] (6) at (90+4*72:1) {};
 \node [black,right] at (6.east) [yshift = -5] {$6$};

 \draw [-,semithick] (0) to (4);
 \draw [-,semithick] (4) to (3);
 \draw [-,semithick] (3) to (2);
 \draw [-,semithick] (2) to (1);
 \draw [-,semithick] (1) to (0);

 \draw [-,semithick] (5) to (8);
 \draw [-,semithick] (5) to (7);
 \draw [-,semithick] (9) to (6);
 \draw [-,semithick] (9) to (7);
 \draw [-,semithick] (6) to (8);

 \draw [-,semithick] (0) to (5);
 \draw [-,semithick] (4) to (9);
 \draw [-,semithick] (3) to (8);
 \draw [-,semithick] (2) to (7);
 \draw [-,semithick] (1) to (6);

\end{tikzpicture}

 \caption{A graph representable by a $2$-uniform word and the Petersen graph}
 \label{fig:petersen}
\end{center}
\end{figure}

The notion of a representable graph comes from algebra, where it was used by Kitaev and Seif to study the growth of the free spectrum of the well known {\em Perkins semigroup}  \cite{KS}. There are also connections between representable graphs and robotic scheduling as described by Graham and Zang in \cite{GZ}. Moreover, representable graphs are a generalization of {\em circle graphs}, which was shown by  Halld\'orsson, Kitaev and Pyatkin in \cite{HKP1}, and thus they are interesting from a graph theoretical point of view. Finally, representable graphs are interesting from a combinatorics on words point of view as they deal with the study of alternations in words.

Not all graphs are representable. Examples of minimal (with respect to the number of nodes) representable graphs given by Kitaev and Pyatkin in \cite{KP} are presented in Fig.~\ref{fig:minimalnonrep}.

\begin{figure}[h]
\begin{center}
\begin{tikzpicture}[place/.style = {circle,draw = black!50,fill = gray!20,thick}, auto]

 \node [place] (0) at (90:1)      {};
 \node [place] (4) at (90+72:1)   {};
 \node [place] (3) at (90+2*72:1) {};
 \node [place] (2) at (90+3*72:1) {};
 \node [place] (1) at (90+4*72:1) {};

 \node [place] (5) at (0,0)      {};

 \draw [-,semithick] (0) to (4);
 \draw [-,semithick] (4) to (3);
 \draw [-,semithick] (3) to (2);
 \draw [-,semithick] (2) to (1);
 \draw [-,semithick] (1) to (0);

 \draw [-,semithick] (0) to (5);
 \draw [-,semithick] (1) to (5);
 \draw [-,semithick] (2) to (5);
 \draw [-,semithick] (3) to (5);
 \draw [-,semithick] (4) to (5);

\end{tikzpicture}\quad
\begin{tikzpicture}[place/.style = {circle,draw = black!50,fill = gray!20,thick}, auto]

 \node [place] (1) at (90:2)      {};
 \node [place] (2) at (90+120:2)   {};
 \node [place] (3) at (90+2*120:2) {};

 \node [place] (4) at (90-17:1)      {};
 \node [place] (5) at (90+120-17:1)   {};
 \node [place] (6) at (90+2*120-17:1) {};

 \node [place] (0) at (0,0)      {};

 \draw [-,semithick] (1) to (2);
 \draw [-,semithick] (2) to (3);
 \draw [-,semithick] (3) to (1);

 \draw [-,semithick] (4) to (5);
 \draw [-,semithick] (5) to (6);
 \draw [-,semithick] (6) to (4);

 \draw [-,semithick] (1) to (4);
 \draw [-,semithick] (2) to (5);
 \draw [-,semithick] (3) to (6);

 \draw [-,semithick] (0) to (1);
 \draw [-,semithick] (0) to (2);
 \draw [-,semithick] (0) to (3);
 \draw [-,semithick] (0) to (4);
 \draw [-,semithick] (0) to (5);
 \draw [-,semithick] (0) to (6);

\end{tikzpicture}\quad
\begin{tikzpicture}[place/.style = {circle,draw = black!50,fill = gray!20,thick}, auto]

 \node [place] (1) at (90:2)      {};
 \node [place] (2) at (90+120:2)   {};
 \node [place] (3) at (90+2*120:2) {};

 \node [place] (4) at (90-17:1)      {};
 \node [place] (5) at (90+120-17:1)   {};
 \node [place] (6) at (90+2*120-17:1) {};

 \node [place] (0) at (0,0)      {};

 \draw [-,semithick] (1) to (2);
 \draw [-,semithick] (2) to (3);
 \draw [-,semithick] (3) to (1);

 \draw [-,semithick] (6) to (4);

 \draw [-,semithick] (1) to (4);
 \draw [-,semithick] (2) to (5);
 \draw [-,semithick] (3) to (6);

 \draw [-,semithick] (0) to (1);
 \draw [-,semithick] (0) to (2);
 \draw [-,semithick] (0) to (3);
 \draw [-,semithick] (0) to (4);
 \draw [-,semithick] (0) to (5);
 \draw [-,semithick] (0) to (6);

\end{tikzpicture}\quad
\begin{tikzpicture}[place/.style = {circle,draw = black!50,fill = gray!20,thick}, auto]

 \node [place] (1) at (90:2)      {};
 \node [place] (2) at (90+120:2)   {};
 \node [place] (3) at (90+2*120:2) {};

 \node [place] (4) at (90-17:1)      {};
 \node [place] (5) at (90+120-17:1)   {};
 \node [place] (6) at (90+2*120-17:1) {};

 \node [place] (0) at (0,0)      {};

 \draw [-,semithick] (1) to (2);
 \draw [-,semithick] (2) to (3);
 \draw [-,semithick] (3) to (1);


 \draw [-,semithick] (1) to (4);
 \draw [-,semithick] (2) to (5);
 \draw [-,semithick] (3) to (6);

 \draw [-,semithick] (0) to (1);
 \draw [-,semithick] (0) to (2);
 \draw [-,semithick] (0) to (3);
 \draw [-,semithick] (0) to (4);
 \draw [-,semithick] (0) to (5);
 \draw [-,semithick] (0) to (6);

\end{tikzpicture}
 \caption{Minimal non-representable graphs}
 \label{fig:minimalnonrep}
\end{center}
\end{figure}
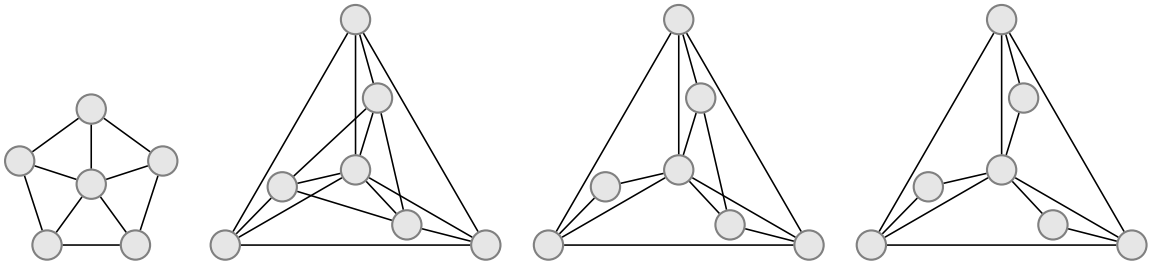

It was remarked in \cite{HKP1} that very little is known about the effect of the line graph operation on the representability of a graph. We attempt to shed some light on this subject by showing that the line graph of the smallest known non-representable graph, the wheel on five vertices, $W_5$, is in fact non-representable. In fact we prove a stronger result, which is that $L(W_n)$ (where $L(G)$ denotes the line graph of $G$) is non-representable for $n \geqslant 4$. From the non-representability of $L(W_4)$ we are led to a more general theorem regarding line graphs. Our main result is that $L^k(G)$, where $G$ is not a cycle, a path or the claw graph, is guaranteed to be non-representable for $k \geqslant 4$.



Although almost all graphs are non-representable (as discussed in \cite{KP}) and even though a criteria in terms of {\em semi-transitive orientations} is given in \cite{HKP1} for a graph to be representable, essentially only two explicit constructions of non-representable graphs are known. Apart from the so-called $co$-$(T_2)$ graph whose non-representability is proved in  \cite{HKP2} in connection with solving an open problem in \cite{KP}, the known constructions of non-representable graphs can be described as follows (note that the property of being representable is hereditary, i.e., it is inherited by all induced subgraphs, thus adding additional nodes to a non-representable graph and connecting them in an arbitrary way to the original nodes will also result in a non-representable graph).

\begin{itemize}
\item Adding an all-adjacent node to a {\em non-comparability graph} results in a non-representable graph (all of the graphs in Fig.~\ref{fig:minimalnonrep} are obtained in this way). This construction is discussed in \cite{KP}.
\item Let $H$ be a 4-chromatic graph with {\em girth} (the length of the shortest cycle) at least 10 (such graphs exist by a theorem of Erd\H{o}s).
For every path of length 3 in $H$ add a new edge connecting the ends of the path. The resulting graph will be non-representable as shown in \cite{HKP1}. This construction gives an example of triangle-free non-representable graphs whose existence was asked for in \cite{KP}.
\end{itemize}

Our results showing that $L(W_n)$, $n \geqslant 4$, and $L(K_{n})$, $n \geqslant 5$, are non-representable give two new constructions of non-representable graphs.


Our main result about repeatedly taking the line graph, shown in Sect.~\ref{iteratingLineGraph}, also gives a new method for constructing non-representable graphs when starting with an arbitrary graph (excluding cycles, paths and the claw graph of course). Since we can start with an arbitrary graph this should also allow one to construct non-representable graphs with desired properties by careful selection of the original graph.


Although we have answered some questions about the line graph operation, there are still open questions related to the representability of the line graph, and in Sect.~\ref{finalRemarks} we list some of these problems.  

\section{Preliminaries on Words and Basic Observations}

\subsection{Introduction to Words}

We denote the set of finite words on an alphabet $\Sigma$ by $\Sigma^*$ and the empty word by $\varepsilon$.

A \emph{morphism} $\varphi$ is a mapping $\Sigma^* \rightarrow \Sigma^*$ that satisfies the property $\varphi(uv)=\varphi(u)\varphi(v)$ for all words $u$, $v$. Clearly, the morphism is completely defined by its action on the letters of the alphabet. The \emph{erasing} of a set $\Sigma \backslash S$ of symbols is a morphism $\epsilon_S : \Sigma^* \rightarrow \Sigma^*$ such that $\epsilon_S(a)=a$ if $a\in S$ and $\epsilon_S(a)=\varepsilon$ otherwise.

A word $u$ \emph{occurs} in a word $v=v_0v_1\ldots v_n$ at the position $m$ and is called a \emph{subword} of $v$ if $u=v_mv_{m+1}\ldots v_{m+k}$ for some $m$, $k$. A subword that occurs at position $0$ in some word is called a \emph{prefix} of that word. A word is \emph{$m$-uniform} if each symbol occurs in it exactly $m$ times. We say that a word is uniform if it is $m$-uniform for some $m$.

Symbols $a,b$ \emph{alternate} in a word $u$ if both of them occur in $u$ and after erasing all other letters in $u$ we get a subword of $abab\ldots$. 

The \emph{alternating graph} $G$ of a word $u$ is a graph on the symbols occurring in $u$ such that $G$ has an edge $(a,b)$ if and only if $a$,$b$ alternate in $u$. A graph $G$ is \emph{representable} if it is the alternating graph of some word $u$. We call $u$ a \emph{representant} of $G$ in this case.

A key property of representable graphs was shown by Kitaev and Pyatkin in \cite{KP}:
\begin{theorem}\label{tRepByUni} Each representable graph has a uniform representant.
\end{theorem}

Assuming uniformity makes dealing with the representant of a graph a much nicer task and plays a crucial role in some of our proofs.

\subsection{Basic Observations}
A \emph{cyclic shift} of a word $u=u_0u_1\ldots u_n$ is the word $Cu=u_1u_2\ldots u_nu_0$.
\begin{proposition}\label{pCShift} Uniform words $u=u_0u_1\ldots u_n$ and $Cu$ have the same alternating graph.
\end{proposition}
\begin{proof} Alternating relations of letters not equal to $u_0$ are not affected by the cyclic shift. Thus we need only prove that $u_0$ has the same alternating relations with other symbols in $Cu$ as it had in $u$.

Suppose $u_0$,$u_i$ alternate in $u$. Due to $u$ being uniform, it must be that $\epsilon_{\{u_0,u_i\}}(u)=(u_0u_i)^m$, where $m$ is the uniform number of $u$. In this case, $\epsilon_{\{u_0,u_i\}}(Cu)=u_i(u_0u_i)^{m-1}u_0$ and hence the symbols $u_0$, $u_i$ alternate in $Cu$.

Suppose $u_0$,$u_i$ do not alternate $u$. Since $u$ is uniform, $u_iu_i$ is a subword of $\epsilon_{\{u_0,u_i\}}(u)$. Also, we know that $u_iu_i$ cannot be the prefix of $u$, so it must occur in $\epsilon_{\{u_0,u_i\}}(Cu)$ too. Hence, $u_0$,$u_i$ do not alternate in $Cu$.
\end{proof}

Taking into account this fact, we may consider representants as cyclic or infinite words in order not to treat differently the end of the word while considering a local part of it.


Let us denote a clique on $n$ vertices by $K_n$. One can easily prove the following proposition.

\begin{proposition}\label{pCliqueRep} An $m$-uniform word that is a representant of $K_{n}$ is a word of the form $v^m$ where $v$ is $1$-uniform word containing $n$ letters.
\end{proposition}

Let us consider another simple graph $C_{n}$ that is a cycle on $n$ vertices.

\begin{lemma}\label{lCycleRep} The word $012\ldots n$ is not a subword of any uniform representant of $C_{n+1}$ with vertices labeled in consecutive order, where $n \geqslant 3$.
\end{lemma}
\begin{proof} Suppose, $u$ is a uniform representant of $C_{n+1}$ and $v=012\ldots n$ is a subword. Due to Proposition~\ref{pCShift} we may assume that $v$ is a prefix of $u$. Define $a_i$ to be the position of the $i$-th instance of $a$ in $u$ for $a \in \{0,1,\ldots,n\}$. Now for all adjacent vertices $a<b$ we have $a_i<b_i<a_{i+1}$ for each $i \geqslant 1$.

Vertices $0$, $2$ are not adjacent in $C_{n+1}$ and so do not alternate in $u$. It follows that there is a $k \geqslant 1$ such that $0_k>2_k$ or $2_{k}>0_{k+1}$.

Suppose $2_k<0_k$. Since $1$,$2$ and $0$,$1$ are adjacent, we have $1_i<2_i$ and $0_i<1_i$ for each $i$. Then we have a contradiction $0_k<1_k<2_k<0_k$.

Suppose $0_{k+1}<2_k$. Since all pairs $j,j+1$ and the pair $n,0$ are adjacent, we have inequalities $j_i<(j+1)_i$ for each $j<n$, $i \geqslant 0$, and $n_i<0_{i+1}$ for each $i \geqslant 0$. Thus we get a contradiction $2_k<3_k<\ldots <n_k<0_{k+1}<2_k$.
\end{proof}

Here we introduce some notation. Let $u$ be a representant of some graph $G$ that contains a set of vertices $S=S_0 \cup S_1 \cup \{a\}$ such that $a\not \in S_0\cup S_1$ and $S_0\cap S_1 = \emptyset$. We use the notation $\forall (a-S_0<S_1-a)$ for the statement ``Between each two neighbor occurrences of $a$ in each $\epsilon_S(C^nu)$, each symbol of $S_0\cup S_1$ occurs once and each symbol of $S_0$ occurs before any symbol of $S_1$'' and the notation $\exists (a-S_0<S_1-a)$ for the statement ``There are two neighbor occurrences of $a$ in some $\epsilon_S(C^nu)$, such that each symbol of $S_0\cup S_1$ occurs between them and each symbol of $S_0$ occurs before any symbol of $S_1$''. Note, that $\forall (a-S_0<S_1-a)$ implies $\exists (a-S_0<S_1-a)$ and is contrary for $\exists (a-S_1<S_0-a)$. The quantifier in this statements operates on pairs of neighbor occurrences of $a$ in all cyclic shifts of the given representant.

\begin{proposition}\label{pAdAndNotAdVertices} Let a word $u$ be a representant of some graph $G$ containing vertices $a,b,c$, where $a,b$ and $a,c$ are adjacent. Then we have
\begin{enumerate}
  \item $b,c$ being not adjacent implies that both of the statements $\exists (a-b<c-a)$, $\exists (a-c<b-a)$ are true for $u$,
  \item $b,c$ being adjacent implies that exactly one of the statements $\forall (a-b<c-a)$, $\forall (a-c<b-a)$ is true for $u$.
\end{enumerate}
\end{proposition}
\begin{proof} (Case 1) Since $a,b$ and $a,c$ alternate, at least one of $\exists (a-b<c-a)$, $\exists (a-c<b-a)$ is true. If only one of them is true for $u$, then $b,c$ alternate in it, which is a contradiction with $b,c$ being not adjacent.

(Case 2) the statement follows immediately from Proposition~\ref{pCliqueRep}.
\end{proof}



\section{Line Graphs of Wheels}\label{nonReprWheels}

The \emph{wheel graph}, denoted by $W_n$, is a graph we obtain from a cycle $C_n$ by adding one external vertex adjacent to every other vertex.

A line graph $L(G)$ of a graph $G$ is a graph on the set of edges of $G$ such that in $L(G)$  there is an edge $(a,b)$ if and only if edges $a,b$ are adjacent in $G$.

\begin{figure}[h]
\begin{center}
\begin{tikzpicture}[place/.style = {circle,draw = black!50,fill = gray!20,thick}, auto]

 \node [place] (0) at (90:2)      {};
 \node [place] (4) at (90+72:2)   {};
 \node [place] (3) at (90+2*72:2) {};
 \node [place] (2) at (90+3*72:2) {};
 \node [place] (1) at (90+4*72:2) {};

 \node [place] (w) at (0,0)      {};

 \draw [-,semithick] (0) to node [auto,swap] {$e_0$} (4);
 \draw [-,semithick] (4) to node [auto,swap] {$e_1$} (3);
 \draw [-,semithick] (3) to node [auto,swap] {$e_2$} (2);
 \draw [-,semithick] (2) to node [auto,swap] {$e_3$} (1);
 \draw [-,semithick] (1) to node [auto,swap] {$e_4$} (0);

 \draw [-,semithick] (0) to node [auto,swap] {$i_4$} (w);
 \draw [-,semithick] (4) to node [auto,swap] {$i_0$} (w);
 \draw [-,semithick] (3) to node [auto,swap] {$i_1$} (w);
 \draw [-,semithick] (2) to node [auto,swap] {$i_2$} (w);
 \draw [-,semithick] (1) to node [auto,swap] {$i_3$} (w);

\end{tikzpicture}\quad\quad\quad
\begin{tikzpicture}[place/.style = {circle,draw = black!50,fill = gray!20,thick}, auto]

 \node [place] (e0) at (0+90:3)    {};
 \node [black,above] at (e0.north) {$e_0$};
 \node [place] (e1) at (72+90:3)   {};
 \node [black,left] at (e1.west) {$e_1$};
 \node [place] (e2) at (144+90:3)  {};
 \node [black,below] at (e2.south) {$e_2$};
 \node [place] (e3) at (216+90:3)  {};
 \node [black,below] at (e3.south) {$e_3$};
 \node [place] (e4) at (288+90:3)  {};
 \node [black,right] at (e4.east) {$e_4$};

 \node [place] (0) at (0+54:1)   {};
 \node [black,above] at (0.north) [xshift = 5] {$i_4$};
 \node [place] (1) at (72+54:1)  {};
 \node [black,above] at (1.north) [xshift = -5] {$i_0$};
 \node [place] (2) at (144+54:1) {};
 \node [black,left] at (2.west) [yshift = -5] {$i_1$};
 \node [place] (3) at (216+54:1) {};
 \node [black,below] at (3.south) {$i_2$};
 \node [place] (4) at (288+54:1) {};
 \node [black,right] at (4.east) [yshift = -5] {$i_3$};

 \draw [-,semithick] (e0) to (e1);
 \draw [-,semithick] (e1) to (e2);
 \draw [-,semithick] (e2) to (e3);
 \draw [-,semithick] (e3) to (e4);
 \draw [-,semithick] (e4) to (e0);

 \draw [-,semithick] (0) to (1);
 \draw [-,semithick] (0) to (2);
 \draw [-,semithick] (0) to (3);
 \draw [-,semithick] (0) to (4);
 \draw [-,semithick] (1) to (2);
 \draw [-,semithick] (1) to (3);
 \draw [-,semithick] (1) to (4);
 \draw [-,semithick] (2) to (3);
 \draw [-,semithick] (2) to (4);
 \draw [-,semithick] (3) to (4);

 \draw [-,semithick] (e0) to (0);
 \draw [-,semithick] (e0) to (1);
 \draw [-,semithick] (e1) to (1);
 \draw [-,semithick] (e1) to (2);
 \draw [-,semithick] (e2) to (2);
 \draw [-,semithick] (e2) to (3);
 \draw [-,semithick] (e3) to (3);
 \draw [-,semithick] (e3) to (4);
 \draw [-,semithick] (e4) to (4);
 \draw [-,semithick] (e4) to (0);

\end{tikzpicture}
 \caption{The wheel graph $W_5$ and its line graph}
 \label{fig:wheel}
\end{center}
\end{figure}
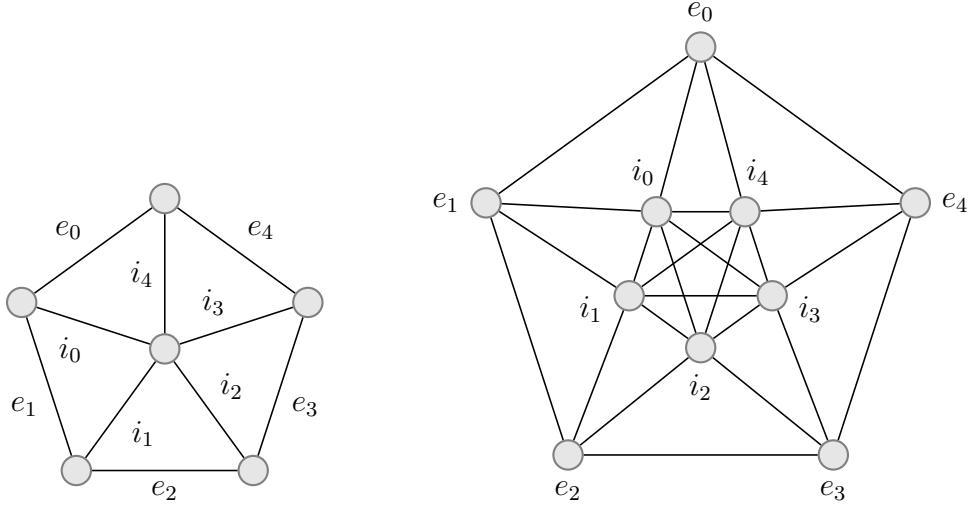

\begin{theorem}\label{wheelNonRep} The line graph $L(W_{n+1})$ is not representable for each $n \geqslant 3$.
\end{theorem}
\begin{proof} Let us describe $L(W_{n+1})$ first. Denote edges of the big (external) cycle of the wheel $W_n$ by $e_{0},e_1,\ldots,e_n$ in consecutive order and internal edges that connect the inside vertex to the big cycle  by $i_0,i_1,\ldots, i_n$ so that an edge $i_{j}$ is adjacent to $e_j$ and $e_{j+1}$ for $0\leqslant j <n$ and $i_n$ is adjacent to $e_n$, $e_0$.

In the line graph $L(W_{n+1})$ the vertices $e_{0},e_1,\ldots,e_n$ form a cycle where they occur consecutively and the vertices $i_0,i_1,\ldots, i_n$ form a clique. In addition, vertices $i_{j}$ are adjacent to $e_j$, $e_{j+1}$ and $i_n$ is adjacent to $e_n$, $e_0$.

Now we suppose that $L(W_{n+1})$ is the alternating graph of some word that, due to Theorem~\ref{tRepByUni}, can be chosen to be uniform and deduce a contradiction with Lemma~\ref{lCycleRep}.

Let $E$ be the alphabet $\{e_j \;:\; 0\leqslant j \leqslant n\}$, $I$ be the alphabet $\{i_j \;:\; 0\leqslant j \leqslant n\}$ and a word $u$ on the alphabet $E\cup I$ be the uniform representant of $L(W_{n+1})$. Due to Proposition~\ref{pCShift}, we may assume $u_0=i_0$.

As we know from Proposition~\ref{pCliqueRep}, the word $\epsilon_I(u)$ is of the form $v^m$, where $v$ is $1$-uniform and $v_0=i_0$. Let us prove that $v$ is exactly $i_0i_1\ldots i_n$ or $i_0i_ni_{n-1}\ldots i_1$.

Suppose there are $1< \ell \not=k\leqslant n$ such that $\epsilon_{\{i_0,i_1,i_\ell,i_k\}}(v)=i_0 i_\ell i_1 i_k$. This implies that the statement $\forall (i_0-i_\ell<i_1<i_k-i_0)$ is true for $u$. The vertex $e_1$ is neither adjacent to $i_\ell$ nor to $i_k$. By Proposition~\ref{pAdAndNotAdVertices} this implies $\exists (i_0-e_1<i_\ell-i_0)$ and $\exists (i_0-i_k<e_1-i_0)$ are true for $u$. Taking into account the previous ``for all'' statement, we conclude that both of $\exists (i_0-e_1<i_1-i_0)$ and $\exists (i_0-i_1<e_1-i_0)$ are true for $u$, which contradicts Proposition~\ref{pAdAndNotAdVertices} applied to $i_0,i_1,e_1$. So, there are only two possible cases, i.e., $v=i_0i_1v_2\ldots v_n$ and $v=i_0v_1\ldots i_1$.

Using the same reasoning on a triple $i_j$, $i_{j+1}$, $e_{j+1}$, by induction on $j\geqslant 1$, we get $v=i_0i_1\ldots i_n$ for the first case and $v=i_0i_ni_{n-1}\ldots i_1$ for the second.

It is sufficient to prove the theorem only for the first case, since reversing a word preserves the alternating relation.

By Proposition~\ref{pAdAndNotAdVertices} exactly one of the statements $\forall (i_0-e_0<e_1-i_0)$, $\forall (i_0-e_1<e_0-i_0)$ is true for $u$. Let us prove that it is the statement $\forall (i_0-e_0<e_1-i_0)$.

Applying Proposition~\ref{pAdAndNotAdVertices} to the clique $\{i_0,i_1,e_1\}$ we have that exactly one of $\forall (i_0-i_1<e_1-i_0)$, $\forall (i_0-e_1<i_1-i_0)$ is true. Applying Proposition~\ref{pAdAndNotAdVertices} to $i_0,i_2,e_1$ we have that both of $\exists (i_0-e_1<i_2-i_0)$ and $\exists (i_0-i_2<e_1-i_0)$ are true. The statement $\forall (i_0-e_1<i_1-i_0)$ contradicts  $\exists (i_0-i_2<e_1-i_0)$ since we have $\forall (i_0-i_1<i_2-i_0)$. Hence $\forall (i_0-i_1<e_1-i_0)$ is true.

Now applying Proposition~\ref{pAdAndNotAdVertices} to $i_0$, $e_0$ and $i_1$ we, in particular, have $\exists (i_0-e_0<i_1-i_0)$. Taking into account $\forall (i_0-i_1<e_1-i_0)$ and Proposition~\ref{pAdAndNotAdVertices} applied to the clique $\{i_0,e_0,e_1\}$ we conclude that $\forall (i_0-e_0<e_1-i_0)$ is true. In other words, between two consecutive $i_0$ in $u$ there is $e_0$ that occurs before $e_1$.

Using the same reasoning, one can prove that the statement $i_n-e_n<e_0-i_n$ and the statements $i_j-e_j<e_{j+1}-i_j$ for each $j<n$ are true for $u$. Let us denote this set of statements by ($\ast$).

The vertex $e_0$ is not adjacent to the vertex $i_{n-1}$ but both of them are adjacent to $i_0$, hence, by Proposition~\ref{pAdAndNotAdVertices}, somewhere in $\epsilon_{\{e_0,i_{n-1},i_0\}}$ the word $i_{n-1}e_0i_0$ occurs. Taking into account what we already proved for $v$, this means that we found the structure $i_0-i_1-\ldots-i_{n-1}-e_0-i_n-i_0-i_1-\ldots-e_n$ in $u$, where symbols of $I$ do not occur in gaps denoted by ``$-$''.

Now inductively applying the statements ($\ast$), we conclude that in $u$ there is a structure $i_{n-1}-e_0-e_1-\ldots-e_n-i_{n-1}$ where no symbol $i_{n-1}$ occurs in the gaps. Suppose the symbol $e_0$ occurs somewhere in the gaps between $e_0$ and $e_n$. Since $e_0$ and $e_n$ are adjacent, that would mean that between two $e_0$ another $e_n$ also occurs and this contradicts the fact that $e_n$ and $i_{n-1}$ are adjacent. Again by induction, one may prove that no symbol of $E$ occurs in the gaps between $e_0$ and $e_n$ in the structure we found. In other words, $e_0e_1\ldots e_n$ occurs in the word $\epsilon_E(u)$ representing the cycle. This results in a contradiction with Lemma~\ref{lCycleRep} which concludes the proof. \end{proof}

\section{Line Graphs of Cliques}\label{nonReprCliques}

\begin{theorem} The line graph $L(K_{n})$ is not representable for each $n \geqslant 5$.
\end{theorem}

\begin{proof} It is sufficient to prove the theorem for the case $n=5$ since, as one can prove, any $L(K_{n \geqslant 5})$ contains an induced $L(K_5)$.

Let $u$ be a representant of $L(K_{5})$ with its vertices labeled as shown in Fig.~\ref{fig:clique}.
Vertices $0,1,a,b$ make a clique in $L(K_{5})$. By applying Propositions~\ref{pCliqueRep} and \ref{pAdAndNotAdVertices} to this clique we see that exactly one of the following statements is true: $\forall (a-\{0,1\}<b-a)$, $\forall (a-b<\{0,1\}-a)$, $\forall (a-x<b<\overline{x}-a)$, where $x\in\{0,1\}$.

\textit{(Case 1)} Suppose $\forall (a-x<b<\overline{x}-a)$ is true. The vertex $3$ is adjacent to $a$, $b$, but not to $x$, $\overline{x}$. Applying Proposition~\ref{pAdAndNotAdVertices} we have that $\exists (a-3<\{0,1\}-a)$ and $\exists (a-\{0,1\}<3-a)$ are true. But between $x$,$\overline{x}$ there is $b$, so we have a contradiction $\exists (a-3<b-a)$, $\exists (a-b<3-a)$ with Proposition~\ref{pAdAndNotAdVertices}.

\textit{(Case 2a)} Suppose $\forall (a-b<0<1-a)$) is true. The vertex $e$ is adjacent with $a$, $0$, but not with $b$, $1$. Applying Proposition~\ref{pAdAndNotAdVertices} we have $\exists (a-e<b-a)$ and $\exists (a-1<e-a)$. Taking into account the case condition, this implies $\exists (a-e<0-a)$ and $\exists (a-0<e-a)$ which is a contradiction.

\textit{(Case 2b)} Suppose $\forall (a-b<1<0-a)$) is true. The vertex $2$ is adjacent with $a$, $1$, but not with $b$, $0$. Applying Proposition~\ref{pAdAndNotAdVertices} we have $\exists (a-2<b-a)$ and $\exists (a-0<2-a)$. Again, taking into account the case condition this implies $\exists (a-2<1-a)$ and $\exists (a-1<2-a)$, which gives a contradiction.

\textit{(Case 3a)} If $\forall (a-0<1<b-a)$) is true, a contradiction follows analogously to Case 2b.  

\textit{(Case 3b)} If $\forall (a-1<0<b-a)$) is true, a contradiction follows analogously to Case 2a.
\end{proof}

\begin{figure}[h]
\begin{center}
\begin{tikzpicture}[place/.style = {circle,draw = black!50,fill = gray!20,thick}, auto]

 \node [place] (0) at (90:2)      {};
 \node [place] (4) at (90+72:2)   {};
 \node [place] (3) at (90+2*72:2) {};
 \node [place] (2) at (90+3*72:2) {};
 \node [place] (1) at (90+4*72:2) {};

 \draw [-,semithick] (0) to node [auto,swap] {$d$} (4);
 \draw [-,semithick] (4) to node [auto,swap] {$e$} (3);
 \draw [-,semithick] (3) to node [auto,swap] {$a$} (2);
 \draw [-,semithick] (2) to node [auto,swap] {$b$} (1);
 \draw [-,semithick] (1) to node [auto,swap] {$c$} (0);

 \draw [-,semithick] (0) to node [auto,swap, yshift = -4] {$2$} (3);
 \draw [-,semithick] (4) to node [auto,swap] {$0$} (2);
 \draw [-,semithick] (3) to node [auto,swap] {$3$} (1);
 \draw [-,semithick] (2) to node [auto,swap, yshift = -4] {$1$} (0);
 \draw [-,semithick] (1) to node [auto,swap] {$4$} (4);

\end{tikzpicture}\quad
\begin{tikzpicture}[place/.style = {circle,draw = black!50,fill = gray!20,thick}, auto]

 \node [place] (e0) at (0+90:3)    {};
 \node [black,above] at (e0.north) {$c$};
 \node [place] (e1) at (72+90:3)   {};
 \node [black,left] at (e1.west) {$d$};
 \node [place] (e2) at (144+90:3)  {};
 \node [black,below] at (e2.south) {$e$};
 \node [place] (e3) at (216+90:3)  {};
 \node [black,below] at (e3.south) {$a$};
 \node [place] (e4) at (288+90:3)  {};
 \node [black,right] at (e4.east) {$b$};


 \node [place] (1) at (0+54:1)   {};
 \node [black,above] at (1.north) [xshift = 5] {$1$};
 \node [place] (4) at (72+54:1)  {};
 \node [black,above] at (4.north) [xshift = -5] {$4$};
 \node [place] (2) at (144+54:1) {};
 \node [black,left] at (2.west) [yshift = -5] {$2$};
 \node [place] (0) at (216+54:1) {};
 \node [black,below] at (0.south) {$0$};
 \node [place] (3) at (288+54:1) {};
 \node [black,right] at (3.east) [yshift = -5] {$3$};

 \draw [-,gray=25!] (e0) to (e1);
 \draw [-,gray=25!] (e1) to (e2);
 \draw [-,thick] (e2) to (e3);
 \draw [-,thick] (e3) to (e4);
 \draw [-,gray=25!] (e4) to (e0);

 \draw [-,thick] (1) to (2);
 \draw [-,thick] (1) to (0);
 \draw [-,gray=25!] (4) to (0);
 \draw [-,gray=25!] (4) to (3);
 \draw [-,gray=25!] (2) to (3);

 \draw [-,gray=25!, bend left=35] (0) to (e1);
 \draw [-,thick] (0) to (e2);
 \draw [-,thick] (0) to (e3);
 \draw [-,thick, bend right=35] (0) to (e4);
 \draw [-,thick, bend left=35] (1) to (e3);
 \draw [-,thick] (1) to (e4);
 \draw [-,gray=25!] (1) to (e0);
 \draw [-,gray=25!, bend right = 35] (1) to (e1);
 \draw [-,gray=25!] (4) to (e0);
 \draw [-,gray=25!] (4) to (e1);
 \draw [-,gray=25!, bend left = 35] (4) to (e4);
 \draw [-,gray=25!, bend right = 35] (4) to (e2);
 \draw [-,gray=25!] (2) to (e1);
 \draw [-,gray=25!] (2) to (e2);
 \draw [-,gray=25!, bend left = 35] (2) to (e0);
 \draw [-,thick, bend right = 35] (2) to (e3);
 \draw [-,thick] (3) to (e3);
 \draw [-,thick] (3) to (e4);
 \draw [-,gray=25!, bend right = 35] (3) to (e0);
 \draw [-,gray=25!, bend left = 35] (3) to (e2);

\end{tikzpicture}
 \caption{The clique $K_5$ and its line graph, where edges mentioned in the proof of Theorem~\ref{nonReprCliques}
 are drawn thicker}
 \label{fig:clique}
\end{center}
\end{figure}
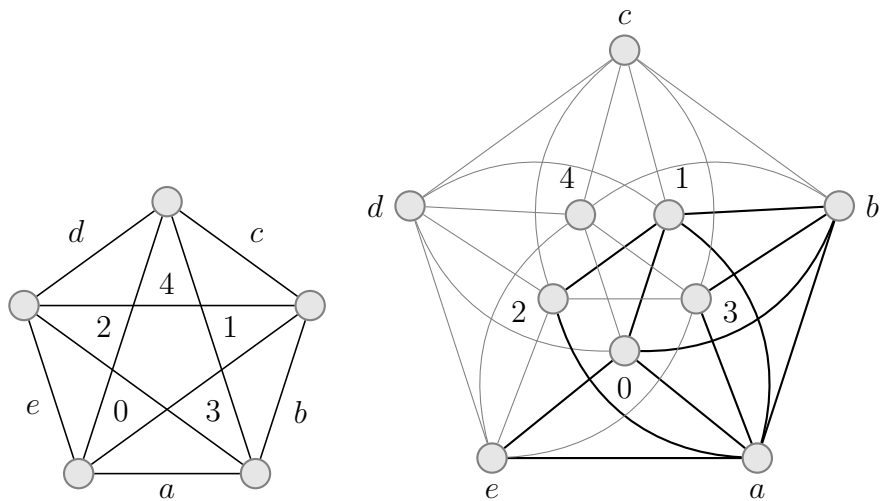

\section{Iterating the Line Graph Construction}\label{iteratingLineGraph}

It was shown by van Rooji and Wilf \cite{WvR} that iterating the line graph operator on most graphs results in a sequence of graphs which grow without bound. The exceptions are cycles, which stay as cycles of the same length, the claw graph $K_{1,3}$, which becomes a triangle after one iteration and then stays that way, and paths, which shrink to the empty graph. This unbounded growth results in graphs that are non-representable after a small number of iterations of the line graph operator since they contain the line graph of a large enough clique. A slight modification of this idea is used to prove our main result. 

\begin{theorem} \label{thm:iterating}
If a connected graph $G$ is not a path, a cycle, or the claw graph $K_{1,3}$, then $L^n(G)$ is not representable for $n \geqslant 4$.
\end{theorem}

\begin{proof}
Note that if $H$ appears as a subgraph of $G$ (not necessarily induced), then $L^n(H)$ is isomorphic to an induced subgraph of $L^n(G)$ for all $n \geqslant 1$. 

We first consider the sequence of graphs in Fig.~\ref{fig:iterating}. All but the leftmost graph are obtained by applying the line graph operator to the previous graph. The last graph in the sequence is $W_4$, and by Theorem~\ref{wheelNonRep}, $L(W_4)$ is non-representable. 
 
Now, let  $G=(V,E)$ be a graph that is not a star and that satisfies the conditions of the theorem.  $G$ contains as a subgraph an isomorphic copy of either the leftmost graph of Fig.~\ref{fig:iterating} or the second graph from the left. Thus $L^3(G)$, or respectively $L^4(G)$, is not representable, since it contains an induced line graph of the wheel $W_4$. 

If $G$ is a star $S_{k \geqslant 4}$ then $L(G)$ is the clique $K_{k}$ and there is an isomorphic copy of the second from the left graph of  Fig.~\ref{fig:iterating} in $G$, and $L^4(G)$ is not representable again.

Note that there is an isomorphic copy of the second graph of Fig.~\ref{fig:iterating} inside the third one. Therefore the same reasoning can be used for $L^{4+k}(G)$ for each $k \geqslant 1$, which concludes the proof.
\end{proof}

\begin{figure}[h]
\begin{center}
\begin{tikzpicture}[place/.style = {circle,draw = black!50,fill = gray!20,thick}, auto]


\draw [->,thick] (3.5,0) to node {$L$} (4.5,0);
\draw [->,thick] (7.5,0) to node {$L$} (8.5,0);
\draw [->,thick] (11.5,0) to node {$L$} (12.5,0);

\begin{scope}
 \node [place] (1) at (90:1)  {};
 \node [place] (2) at (270:1) {};
 \node [place] (3) at (0:1)   {};
 \node [place] (4) at (0:2)   {};
 \node [place] (5) at (0:3)   {};

 \draw [-,semithick] (2) to (3);
 \draw [-,semithick] (3) to (1);
 \draw [-,semithick] (3) to (4);
 \draw [-,semithick] (4) to (5);
\end{scope}

\begin{scope}[xshift=5cm]
 \node [place] (1) at (90:1)  {};
 \node [place] (2) at (270:1) {};
 \node [place] (3) at (0:1)   {};
 \node [place] (4) at (0:2)   {};

 \draw [-,semithick] (1) to (2);
 \draw [-,semithick] (2) to (3);
 \draw [-,semithick] (3) to (1);
 \draw [-,semithick] (3) to (4);
\end{scope}

\begin{scope}[xshift=10cm]

 \node [place] (1) at (90:1)  {};
 \node [place] (2) at (180:1) {};
 \node [place] (3) at (270:1) {};
 \node [place] (4) at (0:1)   {};

 \draw [-,semithick] (1) to (2);
 \draw [-,semithick] (2) to (3);
 \draw [-,semithick] (3) to (4);
 \draw [-,semithick] (4) to (1);
 \draw [-,semithick] (1) to (3);

\end{scope}

\begin{scope}[xshift=14cm]

 \node [place] (1) at (90:1)  {};
 \node [place] (2) at (180:1) {};
 \node [place] (3) at (270:1) {};
 \node [place] (4) at (0:1)   {};
 \node [place] (0) at (0,0)   {};

 \draw [-,semithick] (1) to (2);
 \draw [-,semithick] (2) to (3);
 \draw [-,semithick] (3) to (4);
 \draw [-,semithick] (4) to (1);
 \draw [-,semithick] (0) to (1);
 \draw [-,semithick] (0) to (2);
 \draw [-,semithick] (0) to (3);
 \draw [-,semithick] (0) to (4);
\end{scope}

\end{tikzpicture}
 \caption{Iterating the line graph construction}
 \label{fig:iterating}
\end{center}
\end{figure}
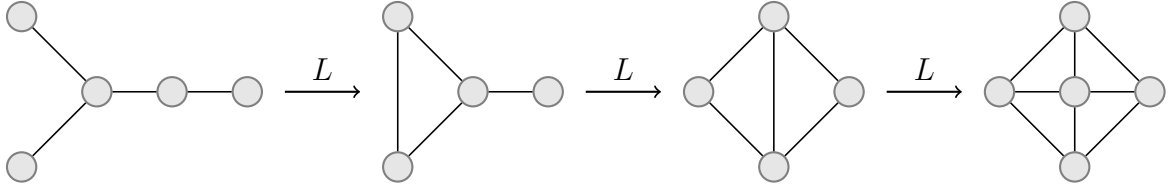

\section{Some Open Problems}\label{finalRemarks}

We have the following open questions.

\begin{itemize}
\item Is the line graph of a non-representable graph always non-representable?
\end{itemize}
Our Theorem~\ref{thm:iterating} shows that for any graph $G$, that is not a path, a cycle, or the claw $K_{1,3}$, the graph $L^{n}(G)$ is non-representable for all $n \geqslant 4$. It might be possible to find a graph $G$ such that $G$ is non-representable while $L(G)$ is.

\begin{itemize}
\item How many graphs on $n$ vertices stay non-representable after at most $i$ iterations, $i = 0,1,2,3,4$?
\end{itemize}
For a graph $G$ define $\xi(G)$ as the smallest integer such that $L^{k}(G)$ is non-representable for all
$k \geqslant \xi(G)$. Theorem~\ref{thm:iterating} shows that $\xi(G)$ is at most $4$, for a graph that
is not a path, a cycle, nor the claw $K_{1,3}$, while paths, cycles and the claw have $\xi(G) = + \infty$.

\begin{itemize}
\item Is there a finite classification of prohibited subgraphs in a graph G determining whether $L(G)$
is representable?
\end{itemize}
There is a classification of prohibited induced subgraphs which determine whether a graph $G$ is the line graph of some other graph $H$. It would be nice to have such a classification, if one exists, to determine if $L(G)$ is representable.

\end{document}